\documentclass[10pt,twoside]{siamltex}
\usepackage{amsfonts,epsfig}
\usepackage{graphicx}
\usepackage{epsfig}

\usepackage{xy}
\input{xy}
\xyoption{all} \xyoption{rotate}

\newtheorem{remark}[theorem]{Remark}

\begin{document}

\bibliographystyle{plain}
\title{
Normal forms of triangles and quadrilaterals up to similarity}

\author{
Peteris\ Daugulis, Vija Vagale\thanks{Department of Mathematics,
Daugavpils University, Daugavpils, LV-5400, Latvia
(peteris.daugulis@du.lv). } }

\pagestyle{myheadings} \markboth{ P.\ Daugulis, V.Vagale }{Normal
forms of triangles} \maketitle

\begin{abstract} In this paper the problem of finding a normal form of triangles and plane
quadrilaterals up to similarity is considered. Several normal
forms for triangles and a normal form for quadrilaterals of
special case are described. Normal forms of simple plane objects
such as triangles can be used in mathematics teaching.
\end{abstract}



\section{Introduction}

Many problems of classical Euclidean geometry explicitly or
implicitly consider objects up to similarity. Understanding and
using similarity is an important geometry competence feature for
schoolchildren.

Recall that two geometric figures $A$ and $B$ are similar if $B$
can be obtained from $A$ after a finite composition of
translations, rotations, reflections and dilations (homotheties).
Similarity is an equivalence relation and thus, for example, the
set of all triangles in a plane is partitioned into similarity
equivalence classes which can be identifies with \sl similarity
types\rm\ of triangles.

In many areas of mathematics objects are studied up to equivalence
relations.  Depending on situation and traditions this is done
explicitly, implicitly or inadvertently. The problem of finding a
distinguished (\sl canonical, normal\rm) representatives of
equivalence classes of objects ir posed. Alternatively, it is the
problem of mapping the quotient set injectively back to the
original set. Let $X$ be a set with an equivalence relation $\sim$
or, equivalently, $R\subseteq X\times X$, denote the equivalence
class of $x\in X$ by $[x]$. Let $\pi:X\rightarrow X/R$ such that
$\pi(x)=[x]$ be the canonical projection map. We call a map
$\sigma: X/R \rightarrow X$ \sl normal object map\rm\ provided 1)
$\sigma$ is injective and 2) $\sigma\circ \pi=id_{X}$. For
example, there are various normal forms of matrices, such as the
Jordan normal form. See \cite{Sh} for examples of normal forms in
algebra and \cite{P} for a related recent work. Normal objects are
designed for educational, pure research (e.g. for classification)
and applied reasons. Normal objects are constructed as objects of
simple, minimalistic design, to show essential properties and
parameters of original objects. Often it is easier to solve a
problem for normal objects first and extend the solution to
atbitrary objects afterwards. Normal objects which are initially
designed for educational, pure research or problem solving
purposes are also used to optimize computations.

In elementary Euclidean geometry normal map approach does not seem
to be popular working with simple discrete objects such as
triangles. This may be related to the traditional dominance of the
synthetic geometry in school mathematics at the expense of the
coordinate/analytic approach. We can pose the problem of
introducing and using normal forms of triangles up to similarity.
This means to describe a set $S$ of mutually non-similar triangles
such that any triangle in the plane would be similar to a triangle
in $S$. We assume that Cartesian coordinates are introduced in the
plane, $S$ is designed using the Cartesian coordinates. For
triangle we offer three normal forms based on side lengths. Using
these normal forms the set of triangle similarity forms is
bijectively mapped to a fixed plane domain bounded by lines and
circles. For these forms two vertices are fixed and the third
vertex belongs to this finite domain, we call them \sl the one
vertex normal forms.\rm\ One vertex normal forms are also
generalized for quadrilaterals. Another normal form for triangles
is based on angles and circumscribed circles. For this form one
constant vertex is fixed on the unit circle and two other variable
vertices also belong to the unit circle, we call this form \sl the
circle normal form.\rm\

These normal forms may be useful in solving geometry problems
involving similarity and teaching geometry. The paper may be
useful for mathematics educators.

\section{Main results}

\subsection{Normal forms of triangles}

\subsubsection{Notations}

Consider $\mathbb{R}^{2}$ with a Cartesian system of coordinates
$(x,y)$ and center $O$. We think of classical triangles as being
encoded by their vertices. Strictly speaking by the triangle
$\triangle XYZ$ we mean the multiset $\{\{X,Y,Z\}\}$ of three
points in $\mathbb{R}^2$ each point having multiplicity at most
$2$. A triangle is called degenerate if points lie on a line.
Given $\triangle ABC$ we denote $\angle BAC=\alpha$, $\angle
ABC=\beta$, $\angle ACB=\gamma$, $|BC|=a$, $|AC|=b$, $|AB|=b$. We
exclude mutisets of type $\{\{XXX\}\}$.

We will use the following affine transformations of
$\mathbb{R}^{2}$: 1) translations, 2) rotations, 3) reflections
with respect to an axis, 4) dilations (given by the rule
$(x,y)\rightarrow (cx,cy)$ for some $c\in \mathbb{R}\backslash
\{0\}$). It is known that these transformations generate the \sl
dilation group\rm\ of $\mathbb{R}^{2}$, denoted by some authors as
$IG(2)$, see \cite{H}, \cite{P}. Two triangles $T_{1}$ and $T_{2}$
as similar if there exists $g\in IG(2)$, such that
$g(T_{1})=T_{2}$ (as multisets). If triangles $T_{1}$ and $T_{2}$
are similar, we write $T_{1}\sim T_{2}$ or $\triangle
X_{1}Y_{1}Z_{1}\sim \triangle X_{2}Y_{2}Z_{2}$.

The point $(x_{1},y_{1})$ is lexicographically smaller than the
point $(x_{2},y_{2})$ and denoted as $(x_{1},y_{1})\prec
(x_{2},y_{2})$ provided ($x_{1}<x_{2}$) or ($x_{1}=x_{2}$ and
$y_{1}<y_{2}$). The lexicographical order of points can be
extended to lexicographical ordering of sequences of points: the
sequence of points $[p_{1},p_{2}]$ is lexicographically smaller
than the sequence $[q_{1},q_{2}]$ denoted by $[p_{1},p_{2}]\prec
[q_{1},q_{2}]$ provided ($p_{1}\prec q_{1}$) or ($p_{1}=q_{1}$ and
$p_{2}\prec q_{2}$).

We use normal letters to denote fixed objects and $\backslash
mathcal$ letters to denote objects as function values.

\subsubsection{The $C$-vertex normal form}\label{1}

A normal form can be obtained by transforming the longest side of
the triangle into a unit interval of the $x$-axis. We call it \sl
the $C$-normal form.\rm\

In this subsection $A=(0,0)$ and $B=(1,0)$.

\begin{definition}

Let $S_{C}\subseteq \mathbb{R}^{2}$ be the domain in the first
quadrant bounded by the lines $y=0$, $x=\frac{1}{2}$ and the
circle $x^2+y^2=1$, see Figure 1.
\begin{center}
\epsfysize=70mm
    \epsfbox{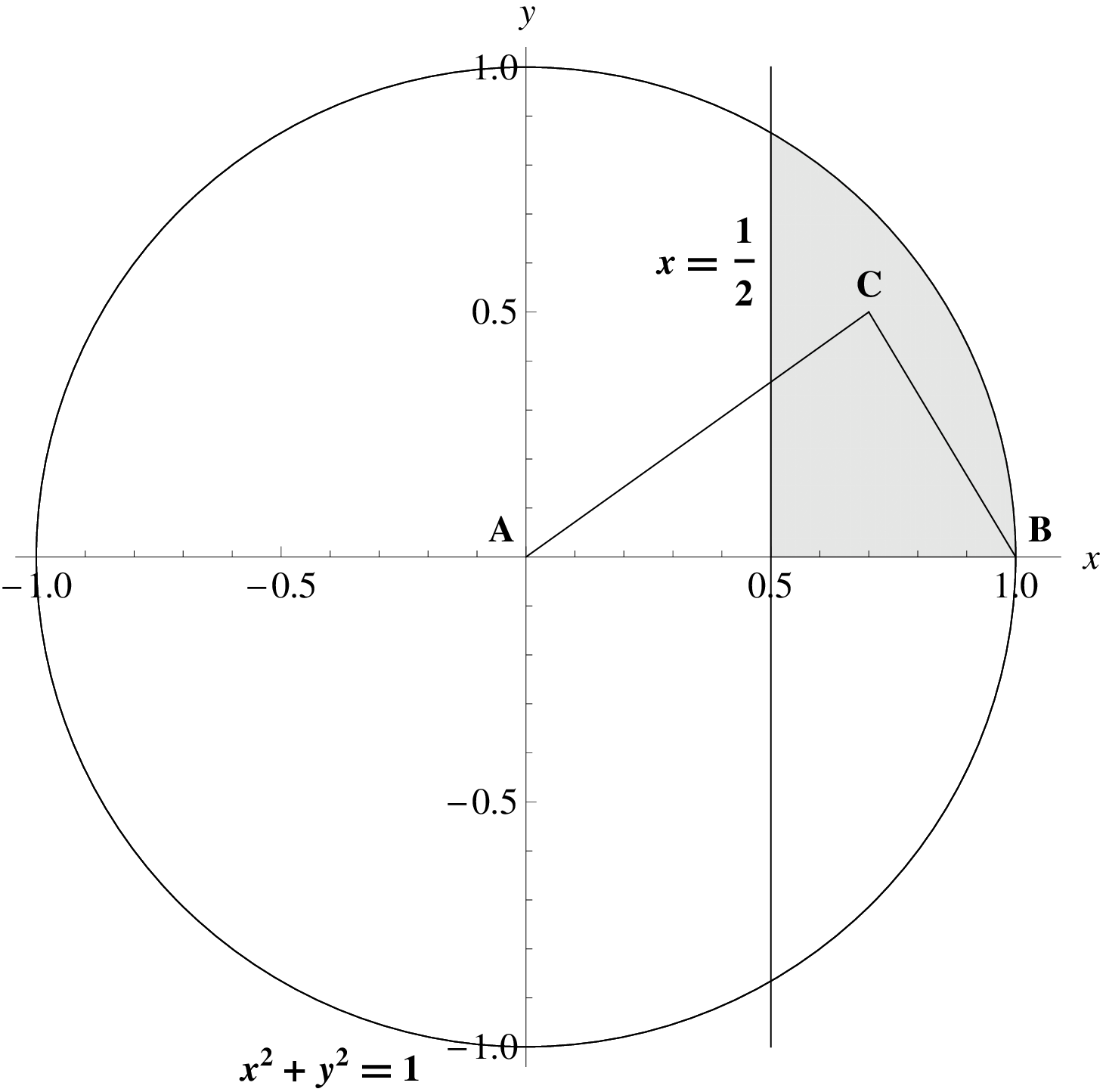}

Fig.1.  - the domain $S_{C}$.
\end{center}

In other terms, $S_{C}$ is the set of solutions of the system of
inequalities
$$
\left\{%
\begin{array}{ll}
    y\ge 0 \\
    x\ge \frac{1}{2}\\
    x^2+y^2\le 1. \\
\end{array}%
\right.
$$
\end{definition}

\begin{theorem} Every triangle $UVW$ (including degenerate triangles)
in $\mathbb{R}^{2}$ is similar to a triangle $AB\mathcal{C}$,
where $A=(0,0)$, $B=(1,0)$ and $\mathcal{C}\in S_{C}$.

\end{theorem}

\begin{proof} Let $\triangle UVW$ has side lengths $a,b,c$
satisfying $a\le b\le c$. Perform the following sequence of
transformations:
\begin{enumerate}

\item translate and rotate the triangle so that the longest side
is on the $x$-axis, one vertex has coordinates $(0,0)$ and another
vertex has coordinates $(c,0)$, $c>0$;

\item if the third vertex has negative $y$-coordinate, reflect the
triangle with respect to the $x$-axis;

\item do the dilation with coefficient $\frac{1}{c}$, note that
the vertices on the $x$-axis have coordinates $(0,0)$ and $(1,0)$,
the third vertex has coordinates $(x'_{C},y'_{C})$, where
$x'^2_{C}+y'^2_{C}\le 1$ and $(x'_{C}-1)^2+y'^2_{C}\le 1$;

\item if $x'_{C}<\frac{1}{2}$, then reflect the triangle with
respect to the line $x=\frac{1}{2}$, denote the third vertex by
$\mathcal{C}=(x_{C},y_{C})$, by construction we have that
$\mathcal{C}\in S_{C}$.

\end{enumerate}

The image of the initial triangle $\triangle UVW$ is the triangle
$AB\mathcal{C}$, where $\mathcal{C}\in S_{C}$. All transformations
preserve similarity type therefore $\triangle UVW\sim \triangle
AB\mathcal{C}$.
\end{proof}

\begin{theorem} If $C_{1}\in S_{C}$, $C_{2}\in S_{C}$ and $C_{1}\neq C_{2}$, then $\triangle ABC_{1}\not\sim\triangle ABC_{2}$.

\end{theorem}

\begin{proof} If $\angle C_{1}AB=\angle C_{2}AB$ and $C_{1}\neq C_{2}$, then $\angle C_{1}BA\neq \angle
C_{2}BA$. By equality of angles for similar triangles it follows
that $\triangle ABC_{1}\not\simeq\triangle ABC_{2}$.

Let $\angle C_{1}AB\neq \angle C_{2}AB$. The angle $C_{i}AB$ is
the smallest angle in $\triangle ABC_{i}$. By equality of angles
for similar triangles it again follows that $\triangle
ABC_{1}\not\simeq\triangle ABC_{2}$.

\end{proof}

\begin{definition} A point $\mathcal{C}\in S_{C}$ such that $\triangle AB\mathcal{C}\sim \triangle
UVW$ is called \sl the $C$-normal point\rm\ of $\triangle UVW$.

\end{definition}

\begin{definition} The $C$-vertex normal form of $\triangle UVW$ is
$\triangle AB\mathcal{C}$, where $\mathcal{C}\in S_{C}$ is the
$C$-normal point of $\triangle UVW$.

\end{definition}

\begin{remark} Denote by $R_{C}$ the intersection of the circle
$(x-\frac{1}{2})^2+y^2=(\frac{1}{2})^2$ and $S_{C}$. Points of
$R_{C}$ correspond to right angle triangles. Points below and
above $R_{C}$ correspond to, respectively, obtuse and acute
triangles, see Figure 2.

Points on the intersection of the line $x=\frac{1}{2}$ and $S_{C}$
correspond to isosceles obtuse triangles. Points on the
intersection of the circle $x^2+y^2=1$ and $S_{C}$ correspond to
isosceles acute triangles. Points in the interior of $S_{C}$
correspond to scalene triangles. The point
$(\frac{1}{2},\frac{\sqrt{3}}{2})$ corresponds to the equilateral
triangle. Points on the intersection of the line $y=0$ and $S_{C}$
correspond to degenerate triangles. $\mathcal{C}=B$ for triangles
having side lengths $0,c,c$.
\end{remark}

\begin{remark} A similar normal form can be obtained reflecting
$S_{C}$ with respect to the line $x=\frac{1}{2}$.

\end{remark}

\begin{center}
    \epsfysize=60mm
    \epsfbox{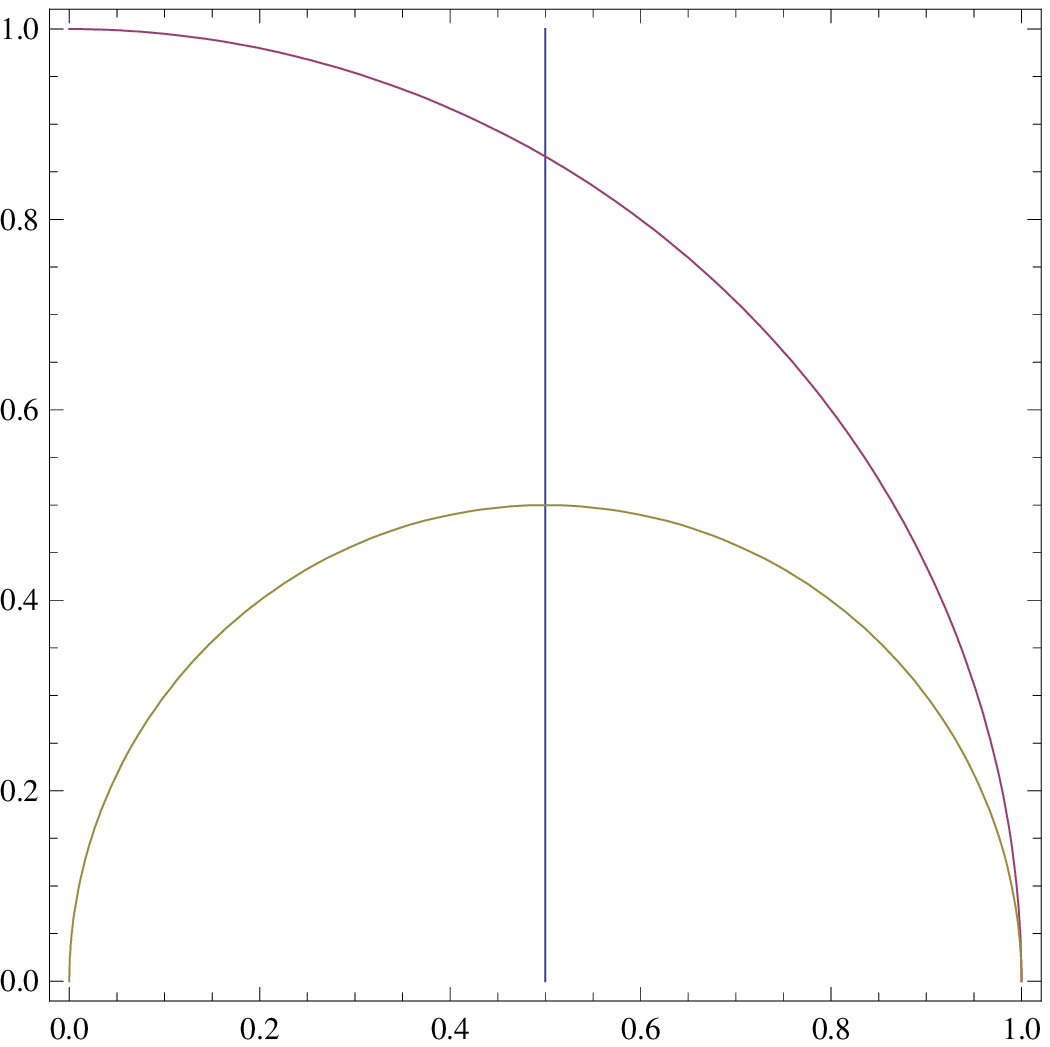}

Fig.2.  - the subdomains of $S_{C}$ corresponding to obtuse and
acute triangles.
    \end{center}

\subsubsection{The $B$-vertex normal form}

Another normal form can be obtained by transforming the median
length side (in the sense of ordering) of the triangle into a unit
interval of the $x$-axis. By analogy it is called \sl the
$B$-normal form.\rm\

In this subsection $A=(0,0)$ and $C=(1,0)$.

\begin{definition}

Let $S_{B}\subseteq \mathbb{R}^{2}$ be the domain in the first
quadrant bounded by the line $y=0$ and the circles $x^2+y^2=1$ and
$(x-1)^2+y^2=1$, see Figure 3.
\begin{center}

\epsfysize=60mm
    \epsfbox{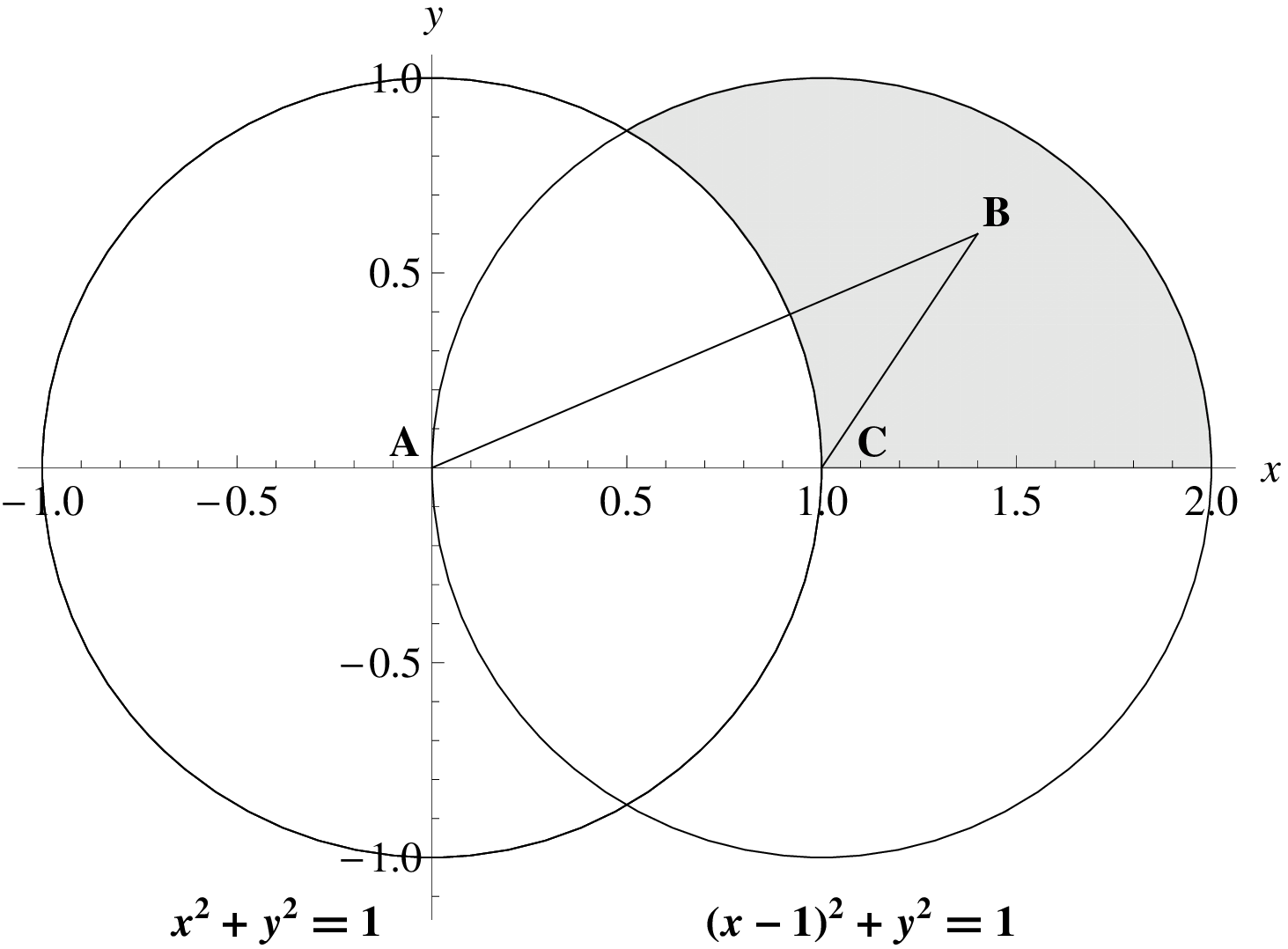}

Fig.3.  - the domain $S_{B}$.
\end{center}

In other terms, $S_{B}$ is the set of solutions of the system of
inequalities
$$
\left\{%
\begin{array}{ll}
    y\ge 0 \\
    x^2+y^2\ge 1\\
    (x-1)^2+y^2\le 1.\\
\end{array}%
\right.
$$
\end{definition}

\begin{theorem} Every triangle $UVW$ (including degenerate triangles)
in $\mathbb{R}^{2}$ is similar to a triangle $A\mathcal{B}C$,
where $A=(0,0)$, $C=(1,0)$ and $\mathcal{B}\in S_{B}$.

\end{theorem}

\begin{proof} Let $\triangle UVW$ has side lengths $a,b,c$
satisfying $a\le b\le c$. Perform the following sequence of
transformations:
\begin{enumerate}

\item translate and rotate the triangle so that the side of length
$b$ is on the $x$-axis, one vertex has coordinates $(b,0)$ and
another vertex has coordinates $(b,0)$, $b>0$, the side of length
$c$ is incident to the vertex $(c,0)$;

\item if the third vertex has negative $y$-coordinate, reflect the
triangle with respect to $x$-axis;

\item do the dilation with coefficient $\frac{1}{b}$, note that
the vertices on the $x$-axis have coordinates $(0,0)$ and $(1,0)$,
at this point the third vertex $\mathcal{B}$ has coordinates
$(x'_{B},y'_{B})$, where $y'_{B}\ge 0$, $x'^2_{B}+y'^2_{B}\ge 1$
or $(x'_{B}-1)^2+y'^2_{B}\le 1$;

\item if $x'_{B}<\frac{1}{2}$, reflect the triangle with respect
to the line $x=\frac{1}{2}$, now the third vertex $\mathcal{B}$
has new coordinates $(x_{B},y_{B})$, where $x_{B}\ge \frac{1}{2}$,
$y_{B}\ge 0$, $(x_{B}-1)^2+y^2_{B}\le 1$.
\end{enumerate}

The image of the initial triangle $\triangle UVW$ is the triangle
$A\mathcal{B}C$, where $\mathcal{B}\in S_{B}$. All transformations
preserve similarity type therefore $\triangle UVW\sim \triangle
A\mathcal{B}C$.
\end{proof}

\begin{theorem} If $B_{1}=(x_{i},y_{i})\in S_{B}$, $B_{2}=(x_{2},y_{2})\in S_{B}$ and $B_{1}\neq B_{2}$, then $\triangle
AB_{1}C\not\sim\triangle AB_{2}C$.

\end{theorem}

\begin{proof} The angle $\angle B_{i}AC$ is
the smallest angle in the triangle $\triangle AB_{i}C$.

If $\angle B_{1}AC\neq \angle B_{2}AC$, then, since these are
smallest angles in the triangles, it follows that $\triangle
AB_{1}C\not\sim \triangle AB_{2}C$.

If $\angle B_{1}AC=\angle B_{2}AC$ and $B_{1}\neq B_{2}$, then
$\angle AB_{1}C\neq \angle AB_{2}C$. $\angle AB_{i}C$ is the
biggest angle in $\triangle AB_{i}C$, therefore $\angle
AB_{1}C\neq \angle AB_{2}C$ implies $\triangle AB_{1}C\not\sim
AB_{2}C$.

\end{proof}

\begin{definition} A point $\mathcal{B}\in S$ such that $\triangle A\mathcal{B}C\sim \triangle
UVW$ is called \sl the $B$-normal point\rm\ of $\triangle UVW$.

\end{definition}

\begin{definition} The $B$-vertex normal form of $\triangle UVW$ is
$\triangle A\mathcal{B}C$, where $\mathcal{B}\in S$ is the
$B$-normal point of $\triangle UVW$.

\end{definition}

\begin{remark} Denote by $R_{B}$ the intersection of the ray
$x=1$, $x\ge 0$ and $S_{B}$. Points of $R_{B}$ correspond to right
angle triangles. Points to the right and left of $R_{B}$
correspond to, respectively, obtuse and acute triangles, see
Figure 4.

\begin{center}
    \epsfysize=60mm
    \epsfbox{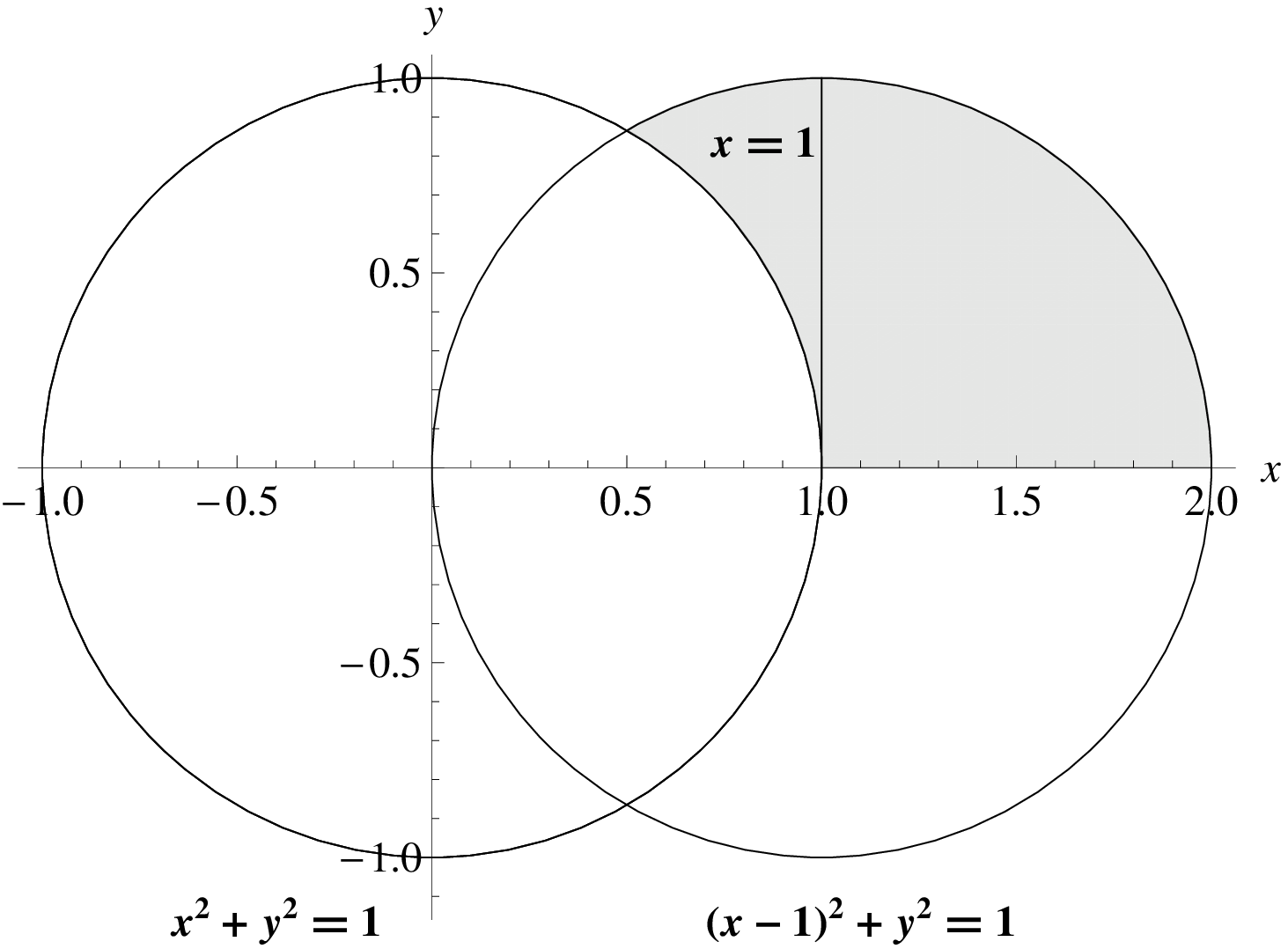}

Fig.4.  - the subdomains of $S_{B}$ corresponding to obtuse and
acute triangles.
    \end{center}

Points on the intersection of the line $x^2+y^2=1$ and $S_{B}$
correspond to isosceles acute triangles. Points on the
intersection of the circle $(x-1)^2+y^2=1$ and $S_{B}$ correspond
to isosceles obtuse triangles. Points in the interior of $S_{B}$
correspond to scalene triangles. The point
$(\frac{1}{2},\frac{\sqrt{3}}{2})$ corresponds to the equilateral
triangle. Points on the intersection of the line $y=0$ and $S_{B}$
correspond to degenerate triangles. $\mathcal{B}=C$ for triangles
having side lengths $0,c,c$.

\end{remark}

\subsubsection{$A$-vertex normal form}

Finally a normal form can be obtained by transforming the shortest
side of the triangle into a unit interval of the $x$-axis. By
analogy it is called \sl the $A$-normal form.\rm\ In this case
again two vertices on the $x$-axis are $(0,0)$ and $(1,0)$, the
domain $S_{A}$ of possible positions of the third vertex is
unbounded.

In this subsection $B=(0,0)$ and $C=(1,0)$.

\begin{definition}

Let $S_{A}\subseteq \mathbb{R}^{2}$ be the unbounded domain in the
first quadrant bounded by the lines $y=0$, $x=\frac{1}{2}$ and the
circle $(x-1)^2+y^2=1$, see Figure 5.

In other terms, $S_{A}$ is the set of solutions of the system of
inequalities
$$
\left\{%
\begin{array}{ll}
    y\ge 0 \\
    x\ge \frac{1}{2}\\
    (x-1)^2+y^2\ge 1.\\
\end{array}%
\right.
$$
\end{definition}

\begin{center}
\epsfysize=60mm
    \epsfbox{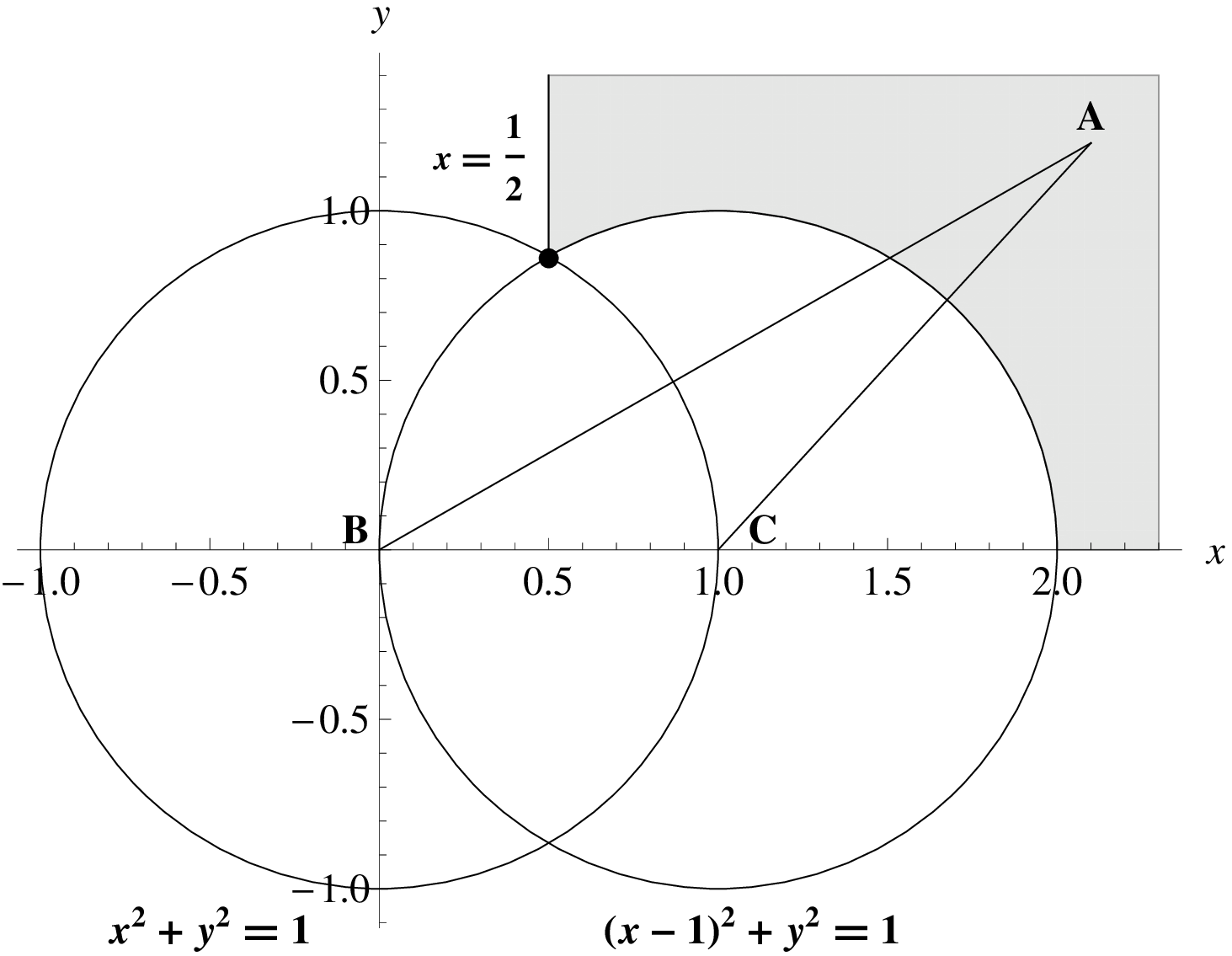}

Fig.5.  - the domain $S_{A}$.
\end{center}

\begin{theorem} Every triangle $UVW$ (including degenerate
triangles but excluding the similarity type having side lengths
$0,c,c$) in $\mathbb{R}^{2}$ is similar to a triangle
$\mathcal{A}BC$, where $B=(0,0)$, $C=(1,0)$ and $\mathcal{A}\in
S_{A}$.

\end{theorem}

\begin{proof} Let $\triangle UVW$ has side lengths $a,b,c$
satisfying $a\le b\le c$. Perform the following sequence of
transformations:
\begin{enumerate}

\item translate and rotate the triangle so that the side of length
$a$ is on the $x$-axis, one vertex has coordinates $(0,0)$ and
another vertex has coordinates $(a,0)$, $a>0$, the side of length
$c$ is incident to the vertex $(0,0)$;

\item if the third vertex has negative $y$-coordinate, reflect the
triangle with respect to the $x$-axis;

\item do the dilation with coefficient $\frac{1}{a}$, note that
the vertices on the $x$-axis have coordinates $(0,0)$ and $(1,0)$;

\item if the third point has the $x$-coordinate less than
$\frac{1}{2}$, reflect the triangle with respect to the line
$x=\frac{1}{2}$, now the third vertex $\mathcal{A}$ has
coordinates $(x_{0},y_{0})$, where $x_{0}\ge \frac{1}{2}$,
$y_{0}\ge 0$, $(x_{0}-1)^2+y^2_{0}\le 1$.

\end{enumerate}

The image of the initial triangle $\triangle UVW$ is the triangle
$\mathcal{A}BC$, where $\mathcal{A}\in S_{A}$. All transformations
preserve similarity type therefore $\triangle UVW\sim \triangle
\mathcal{A}BC$.
\end{proof}

\begin{theorem} Let $B=(0,0)$, $C=(1,0)$. If $A_{1}=(x_{i},y_{i})\in S_{A}$, $A_{2}=(x_{2},y_{2})\in S_{A}$ and $A_{1}\neq A_{2}$, then $\triangle
A_{1}BC\not\sim\triangle A_{2}BC$.

\end{theorem}

\begin{proof} The angle $\angle BCA_{i}$ is
the largest angle in the triangle $\triangle A_{i}BC$.

If $\angle BCA_{1}\neq \angle BCA_{2}$, then since these are
largest angles in the triangles it follows that $\triangle
A_{1}BC\not\sim \triangle A_{2}BC$.

If $\angle BCA_{1}=\angle BCA_{2}$ and $A_{1}\neq A_{2}$, then
$\angle BA_{1}C\neq \angle BA_{2}C$. $BA_{i}C$ is the smallest
angle in $\triangle A_{i}BC$, therefore $A_{1}BC\neq A_{2}BC$
implies $\triangle AB_{1}C\not\sim AB_{2}C$.

\end{proof}

\begin{definition} A point $\mathcal{A}\in S_{A}$ such that $\triangle \mathcal{A}BC\sim \triangle
UVW$ is called \sl the $A$-normal point\rm\ of $\triangle UVW$.

\end{definition}

\begin{definition} The $A$-vertex normal form of $\triangle UVW$ is
$\triangle \mathcal{A}BC$, where $\mathcal{A}\in S_{A}$ is the
$A$-normal point of $\triangle UVW$.

\end{definition}

\begin{remark} Denote by $R_{A}$ the intersection of the ray
$x=1$, $x\ge 0$ and $S_{A}$. Points of $R_{A}$ correspond to right
angle triangles. Points to the right and left of $R_{A}$
correspond to, respectively, obtuse and acute triangles, see
Figure 6.

\begin{center}
    \epsfysize=60mm
    \epsfbox{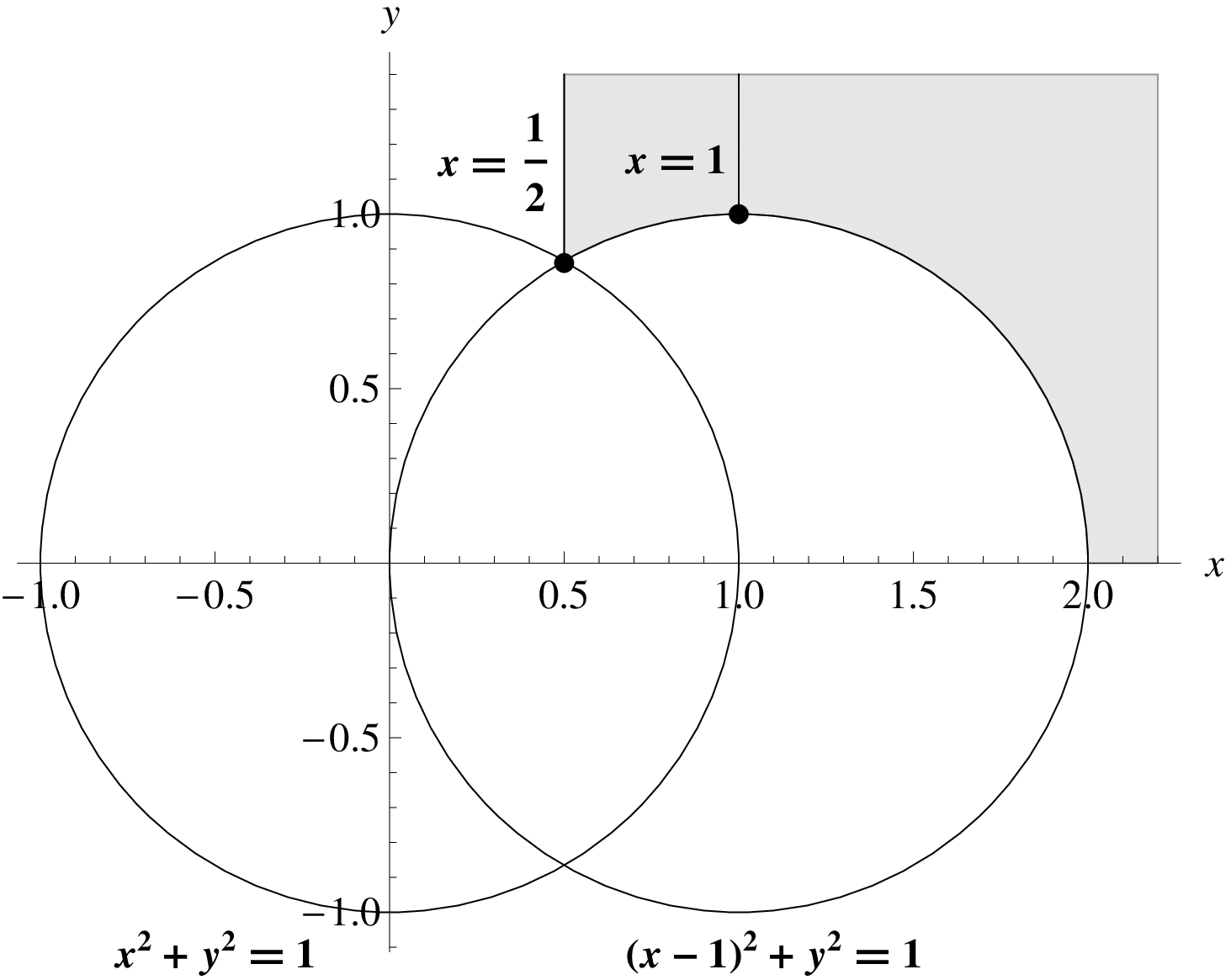}

Fig.6.  - the subdomains of $S_{A}$ corresponding to obtuse and
acute triangles.
    \end{center}

Points on the intersection of the line $x=\frac{1}{2}$ and $S_{A}$
correspond to isosceles acute triangles. Points on the
intersection of the circle $(x-1)^2+y^2=1$ and $S_{A}$ correspond
to isosceles obtuse triangles. Points in the interior of $S_{A}$
correspond to scalene triangles. The point
$(\frac{1}{2},\frac{\sqrt{3}}{2})$ corresponds to the equilateral
triangle. Points on the intersection of the line $y=0$ and $S_{A}$
correspond to degenerate triangles excluding the similarity type
with side lengths $0,c,c$, which corresponds to the point at
infinity. In contrast to the $C$-vertex and $B$-vertex normal
forms triangles for the $A$-vertex normal form are not bounded.

\end{remark}

\subsubsection{The circle normal form}

Consider $\mathbb{R}^{2}$ with a Cartesian system of coordinates
$(x,y)$ and center $O$.  We also consider polar coordinates
$[r,\varphi]$ introduced in the standard way: the polar angle
$\varphi$ is measured from the positive $x$-axis going
counterclockwise.

Note that angles $\alpha,\beta,\gamma$ of a nondegenerate triangle
such that $\alpha\le \beta\le \gamma$ satisfy the system of
inequalities
$$
\left\{%
\begin{array}{ll}
    0<\alpha\le \frac{\pi}{3},\\
    \alpha\le \beta\le \frac{\pi-\alpha}{2}.\\
\end{array}%
\right.
$$

Similarity types of nondegenerate triangles are parametrized by
one point in the domain in $(\alpha,\beta)$-plane determined by
the system
$$
\left\{%
\begin{array}{ll}
    \alpha>0,\\
    \beta\ge\alpha,\\
    \beta\le\frac{\pi}{2}-\frac{\alpha}{2}.\\
\end{array}%
\right.
$$

See Fig.7.

\begin{center}
    \epsfysize=50mm
    \epsfbox{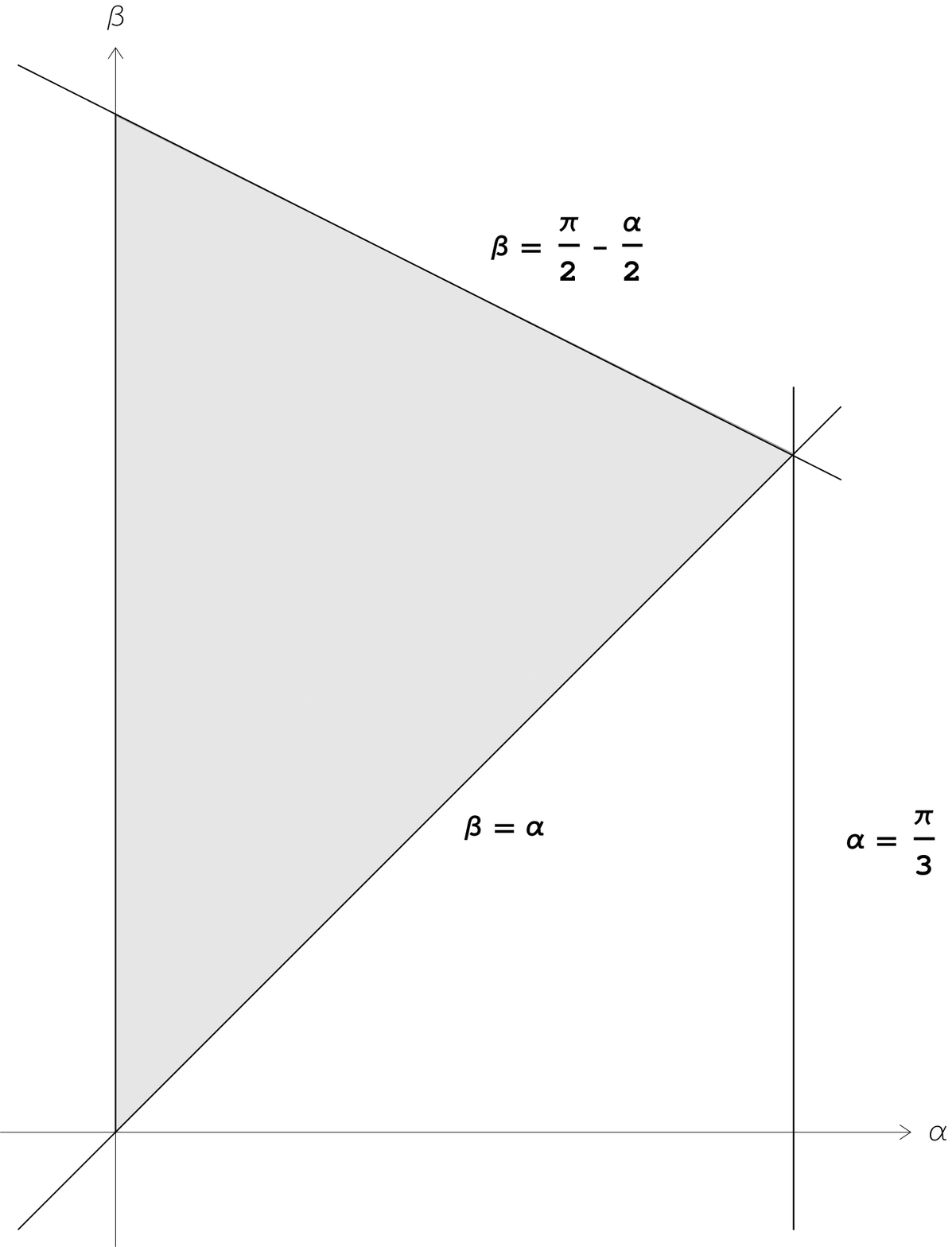}

Fig.7.  - parametrization of similarity types by $(\alpha,\beta)$.
    \end{center}

For the normal form described in this subsection the vertex with
the biggest angle will be fixed at $(1,0)$, to be consistent with
previous notations we define $C=(1,0)$. For this normal form only
nondegenerate triangles are considered.

In this case normal form triangles are inscribed in the unit
circle $\mathbb{U}=\{x^2+y^2=1\}$ having $C$ as one of the
vertices.

\begin{definition} A triangle $\triangle ABC$ inscribed in $\mathbb{U}$ is called \sl normal circle
triangle\rm\ if

\begin{enumerate}

\item $0\le \alpha\le \frac{\pi}{3}$,

\item $\alpha\le \beta\le \frac{\pi}{2}-\frac{\alpha}{2}$,

\item $C=(1,0)$,

\item the point $A$ is above $x$-axis,

\item the point $B$ is below $x$-axis.
\end{enumerate}

\end{definition}

\begin{remark} For a normal triangle $\triangle ABC$ we have that
$\alpha\le \beta \le \gamma$.

\end{remark}

\begin{remark} A normal triangle with angles $\alpha\le \beta\le \gamma$ can be constructed in the
following way:

\begin{enumerate}

\item choose a point $B$ below the $y$-axis with the argument
equal to $2\alpha$, where $0\le 2\alpha\le \frac{2\pi}{3}$;

\item draw the bisector of $\angle BOC$, denote the intersection
of this bisector with the arc $BC$ having angle $2\pi-2\alpha$ by
$D$;

\item find the point $\widetilde{B}$ which is symmetric to $B$
with respect to the $x$-axis;

\item  choose a point $A$ in the shorter arc $\widetilde{B}D$.

\end{enumerate}

See Figure 8.
\begin{center}
    \epsfysize=70mm
    \epsfbox{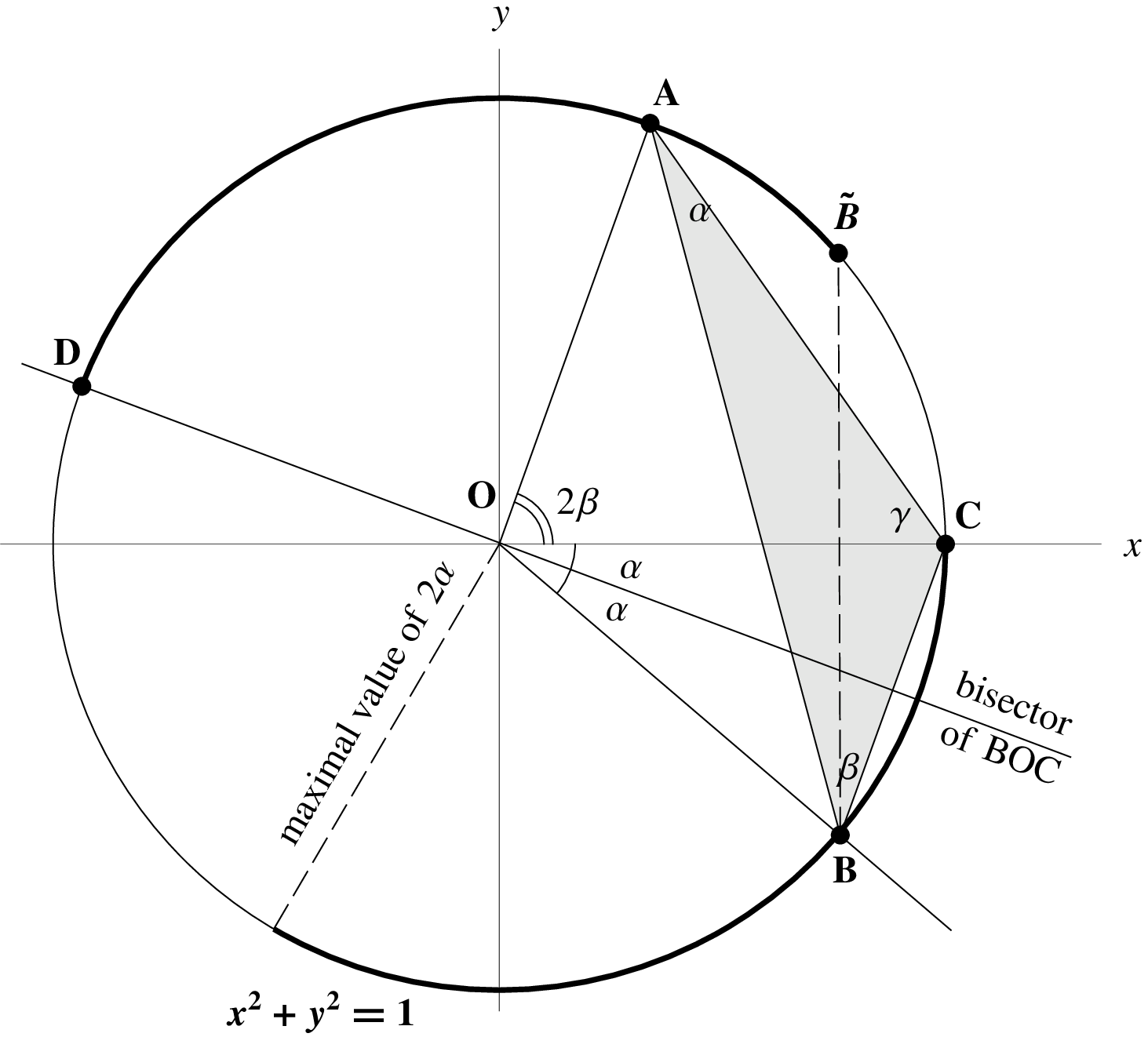}

Fig.8.  - construction of a normal circle triangle.
    \end{center}

\end{remark}

\begin{theorem} For every nondegenerate triangle $\triangle UVW$ there exists a
normal circle triangle $\triangle \mathcal{A}\mathcal{B}C$ such
that $\triangle UVW\sim \triangle \mathcal{AB}C$.

\end{theorem}

\begin{proof} Suppose $\triangle UVW$ has angles $\alpha\le \beta\le \gamma$.
Let $\mathcal{B}\in \mathbb{U}$ be the point with polar
coordinates $[1,-2\alpha]$. Let $\mathcal{A}\in \mathbb{U}$ be the
point with polar coordinates $[1,2\beta]$. Then since $\triangle
\mathcal{AB}C$ is inscribed in $\mathbb{U}$ we have that $\angle
\mathcal{BA}C=\alpha$, $\angle \mathcal{AB}C=\beta$ and thus
$\triangle \mathcal{AB}C\sim \triangle UVW$.

\end{proof}

\begin{theorem} Let $\triangle A_{1}BC_{1}$ and $A_{2}BC_{2}$ be
two distinct normal circle triangles: $A_{1}\neq A_{2}$ or
$B_{1}\neq B_{2}$. Then $\triangle A_{1}BC_{1}\not\sim\triangle
A_{2}BC_{2}$.

\end{theorem}

\begin{proof} If $A_{1}\neq A_{2}$, then $\angle B_{1}A_{1}C\neq \angle
B_{2}A_{2}C$. The angle $B_{i}A_{i}C$ is the smallest angle of
$\triangle A_{i}B_{i}C$. We have that $\angle B_{1}A_{1}C\neq
\angle B_{2}A_{2}C$ implies $\triangle
A_{1}B_{1}C\not\sim\triangle A_{2}B_{2}C$.

If $B_{1}\neq B_{2}$ and $A_{1}=A_{2}$, then $\angle
A_{1}CB_{1}\neq \angle A_{2}CB_{2}$. The angle $A_{i}CB_{i}$ is
the largest angle of $\triangle A_{i}B_{i}C$. We have that $\angle
A_{1}CB_{1}\neq \angle A_{2}CB_{2}$ in this case implies
$\triangle A_{1}B_{1}C\not\sim\triangle A_{2}B_{2}C$.
\end{proof}

\begin{remark} The only isosceles normal triangles are normal triangles of type $\triangle B\widetilde{B}C$ and $\triangle
BDC$. Right normal triangles are normal triangles with $AB$
passing through $O$. Acute/obtuse normal triangles as normal
triangles with $O$ inside/outside $\triangle ABC$. In contrast to
the one vertex normal forms triangles for the circle normal form
are unbounded from below.

\end{remark}

\begin{remark} Other normal forms of this type can be designed
choosing another point instead of $(1,0)$ and rearranging triangle
points.

\end{remark}

\subsubsection{Conversions}\

\begin{definition} Given a triangle with side lengths $a,b,c$
define $N_{X}(a,b,c)$ to be the Cartesian plane coordinates of the
$X$-normal point ($X\in \{A,B,C\}$) corresponding to this
triangle. Note that $N_{X}$ is a symmetric function. We can also
think of arguments of $N_{X}$ as multisets and think that
$N_{X}(a,b,c)=N_{X}(L)$, where $L$ is the multiset
$\{\{a,b,c\}\}$.

\end{definition}

\begin{proposition} Let $\triangle ABC$ has side lengths $a\le b\le c$.

Then

\begin{enumerate}

\item $N_{C}(a,b,c)=\Big(\frac{-a^2+b^2+c^2}{2c^2},
\frac{\sqrt{-a^4-b^4-c^4+2(a^2b^2+a^2c^2+b^2c^2)}}{2c^2}\Big)$;

\item $N_{B}(a,b,c)=\Big(\frac{-a^2+b^2+c^2}{2b^2},
\frac{\sqrt{-a^4-b^4-c^4+2(a^2b^2+a^2c^2+b^2c^2)}}{2b^2}\Big)$;

\item $N_{A}(a,b,c)=\Big(\frac{a^2-b^2+c^2}{2a^2},
\frac{\sqrt{-a^4-b^4-c^4+2(a^2b^2+a^2c^2+b^2c^2)}}{2a^2}\Big)$.

\end{enumerate}

\end{proposition}

\begin{proof}
1. Translate, rotate and reflect $\triangle ABC$ so that
$A=(0,0)$, $B=(c,0)$ and $C=(x,y)$ is in the first quadrant. For
$(x,y)$ we have the system $$\left\{%
\begin{array}{ll}
    x^2+y^2=b^2\\
    (c-x)^2+y^2=a^2\\
\end{array}%
\right.    $$

and find $$\left\{%
\begin{array}{ll}
    x=\frac{-a^2+b^2+c^2}{2c}\\
    y=\frac{\sqrt{-a^4-b^4-c^4+2(a^2b^2+a^2c^2+b^2c^2)}}{2c}\\
\end{array}%
\right.    $$

After the dilation by coefficient $\frac{1}{c}$ we get the given
formula.

 2. and 3. proved are in a similar way.
\end{proof}

\begin{proposition}

Let a triangle $T$ have angles $\alpha\le\beta\le\gamma$.

Then
\begin{enumerate}
\item its $C$-normal point is $N_{C}(\frac{\sin\alpha}{\sin
\gamma},\frac{\sin\beta}{\sin\gamma},1)$;

\item if $T$ has the $C$-normal point $(x,y)$, then it has angles
$\alpha=\arctan{\frac{y}{x}}$, $\beta=\arctan{\frac{y}{1-x}}$,
$\gamma=\pi-\arctan{\frac{y}{x}}-\arctan{\frac{y}{1-x}}$.

\item its $B$-normal point is $N_{B}(\frac{\sin\alpha}{\sin
\beta},\frac{\sin\gamma}{\sin\beta},1)$;

\item if $T$ has the $B$-normal point $(x,y)$, then it has angles
$\alpha=\arctan\frac{y}{x}$,
$\beta=-\arctan\frac{y}{x}+\arctan\frac{y}{x-1}$,
$\gamma=\pi-\arctan\frac{y}{x-1}$;

\item its $A$-normal point is $N_{A}(\frac{\sin\beta}{\sin
\alpha},\frac{\sin\gamma}{\sin\alpha},1)$;

\item if $T$ has the $A$-normal point $(x,y)$, then it has angles
$\alpha=-\arctan\frac{y}{x}+\arctan\frac{y}{x-1}$,
$\beta=\arctan\frac{y}{x}$, $\gamma=\pi-\arctan\frac{y}{x-1}$.
\end{enumerate}
\end{proposition}

\begin{proof}

1. Let $\triangle ABC$ be the $C$-normal triangle with angles
$\alpha\le \beta \le \gamma$, i.e. $|AB|=1$. By the law of sines
we have  $b=|AC|=\frac{\sin \beta}{\sin \gamma}$ and $a=\frac{\sin
\alpha}{\sin\gamma}$. By definition $C$ has coordinates
$N_{C}(\frac{\sin\alpha}{\sin\gamma},\frac{\sin\beta}{\sin\gamma},1)$.

2. Let $CD$ be a height of $\triangle ABC$. Formulas for angles
are obtained considering $\triangle ACD$ and $\triangle BCD$.

3.,4.,5.,6. ar proved similarly.
\end{proof}

\subsection{Normal forms of quadrilaterals}

In this subsection we consider mutisets of $4$ points in a plane.
A multiset of $4$ points can be interpreted as a quadrilateral. We
exclude the case of one point of multiplicity $4$. The multiset
$Q=\{\{X,Y,Z,T\}\}$ is also denoted as $\Box XYZT$. We define
$Q_{1}\sim Q_{2}$ provided there is an element of the dilation
group $g$ such that $g(Q_{1})=Q_{2}$.

A set of $4$ points defines a set of $6$ distances between these
points. Choosing any two points we can translate, rotate, reflect
and dilate the given $4$-point configuration so that the chosen
two points have coordinates $A=(0,0)$ and $B=(1,0)$. Different
normal forms can be obtained choosing pairs with different
relative metric properties. In this paper we consider only the
simplest case - two points having the maximal distance are mapped
to the $x$-axis.

\subsubsection{Longest distance normal form} Suppose we are given  a quadrilateral $\Box XYZT$ such that $|XY|\ge |XZ|$,
$|XY|\ge |XT|$, $|XY|\ge |YZ|$, $|XY|\ge |YT|$, $|XY|\ge |ZT|$. We
map $X$ and $Y$ by a dilation to the $x$-axis (to $A=(0,0)$ and
$B=(1,0)$) and determine what are positions of the $2$ remaining
vertices $C$ and $D$ so that $\Box XYZT \sim \Box ABCD$.

%

\begin{definition} Let $p\in \mathbb{R}^2$.  The mapping of $p$ by reflections of $p$ with respect to
the $x$-axis and the line $x=\frac{1}{2}$ to the domain $y\ge 0$,
$x\ge \frac{1}{2}$ is denoted by $p_{s}$.

\begin{definition} Let $p, p'\in \mathbb{R}^2$. We say that
$p$ is \sl quasilexicographically\rm\ smaller or equal to $p'$,
denoted by $p\vartriangleleft p'$, provided $p_{s}\prec p'_{s}$ or
$p_{s}=p'_{s}$. Given two pairs $[p,q]$ and $[p',q']$ we define
$[p,q]\vartriangleleft [p',q']$ provided ($p_{s}\prec p'_{s}$) or
($p_{s}=p'_{s}$ and $q\vartriangleleft q'$).

\end{definition}

\end{definition}


\begin{definition}
Let $S_{D}(x_{0},y_{0})\subseteq \mathbb{R}^2$ with
$(x_{0},y_{0})\in S_{C}$ (for the definition of $S_{C}$ see
section  \ref{1}) be the set of solutions of the following system
of inequalities:

\begin{equation}\label{2}
\left\{%
\begin{array}{ll}
    x^2+y^2\le 1\\
    (x-1)^2+y^2\le 1\\
    (x-x_{0})^2+(y-y_{0})^2\le 1\\
    |x-\frac{1}{2}|\le |x_{0}-\frac{1}{2}|\\
    if\ |x-\frac{1}{2}|=|x_{0}-\frac{1}{2}|,then\ |y|\le |y_{0}|\\

\end{array}%
\right.
\end{equation}

 See Figure 9.

\begin{center}
    \epsfysize=70mm
    \epsfbox{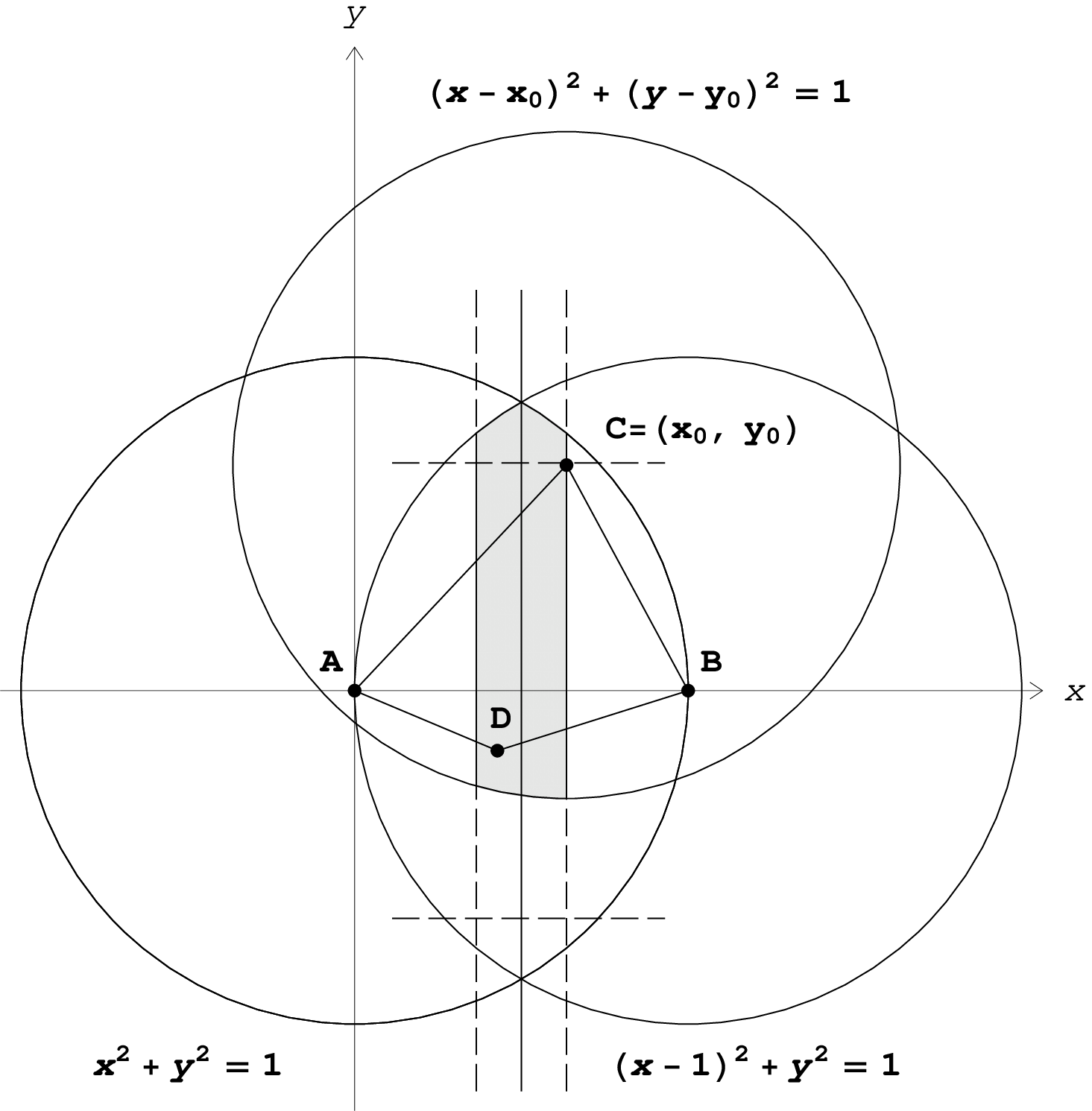}

Fig.9.  - example of the domain $S_{D}$.
    \end{center}
\end{definition}

\begin{remark} Conditions for $p\in S_{D}(x_{0},y_{0})$ consist of two
parts:
\begin{enumerate}
\item distance from $p$ to $A$, $B$ and $(x_{0},y_{0})$ is less
than or equal to $1$;

\item $p_{s}\vartriangleleft (x_{0},y_{0})$.
\end{enumerate}

\end{remark}

\begin{theorem}\label{3} Every $\Box UVWZ$  (including
multisets with multiplicities at most $3$) in $\mathbb{R}^{2}$ is
similar to $\Box AB\mathcal{CD}$, where $A=(0,0)$, $B=(1,0)$,
$\mathcal{C}\in S_{C}$ and $\mathcal{D}\in S_{D}$.
\end{theorem}

\begin{proof} Let $UVWZ$ be a multiset of points in $\mathbb{R}^{2}$ with at least two distinct elements.
Perform the following sequence of transformations:
\begin{enumerate}

\item translate and rotate the plane so that $2$ points with the
longest distance are on the $x$-axis, one vertex has coordinates
$(0,0)$ and another vertex has coordinates $(d,0)$, $d>0$; if
there is more than one possibility to choose two points with the
longest distance then choose this pair so that the remaining pair
is the largest in the quasilexicographic order;

\item do the dilation with coefficient $\frac{1}{d}$, note that
the vertices on the $x$-axis have coordinates $(0,0)$ and $(1,0)$,
suppose the the other two vertices have coordinates
$(x_{C},y_{C})$ and $(x_{D},y_{D})$;

\item if $|x_{C}-\frac{1}{2}|\ne |x_{D}-\frac{1}{2}|$, then put
the point with the maximal $|x-\frac{1}{2}|$ value into $S_{C}$ by
reflections with respect to the $x$-axis and the line
$x=\frac{1}{2}$;

\item if  $|x_{C}-\frac{1}{2}|=|x_{D}-\frac{1}{2}|$, then put the
point with the maximal value of $|y|$ into $S_{C}$ by reflections
with respect to the $x$-axis and the line $x=\frac{1}{2}$;

\item if $|x_{C}-\frac{1}{2}|=|x_{D}-\frac{1}{2}|$ and
$|y_{C}|=|y_{D}|$, then map any of the points into $S_{C}$.

\end{enumerate}

Denote the point which is mapped to $S_{C}$ by this sequence of
transformations by $C=(x_{C},y_{C})$ and the fourth point by
$D=(x_{D},y_{D})$. For any $C=(x_{0},y_{0})\in S_{C}$ we have that
$S_{D}(x_{0},y_{0})\neq \emptyset$.

We check that $D\in S_{D}(x_{C},y_{C})$. From conditions $|AD|\le
1, |BD|\le 1$, $|CD|\le 1$ it follows that $D$ satisfies the first
three inequalities of the system \ref{2}. If $|y_{D}|>|y_{C}|$,
then $|x_{D}-\frac{1}{2}|<|x_{C}-\frac{1}{2}|$ due to the
quasilexicographic order condition.
\end{proof}

\begin{definition} \sl The longest distance normal form\rm\ of $\Box UVWZ$
is $\Box ABCD$ with $A=(0,0)$, $B=(1,0)$, $C=(x_{C},y_{C})\in
S_{C}$ and $D\in S_{D}(x_{C},y_{C})$ constructed according to the
algorithm given in the proof of \ref{3}.
\end{definition}

\begin{proposition} Let $\Box ABC_{1}D_{1}$ and $\Box ABC_{2}D_{2}$ be two
quadrilaterals constructed according to the longest distance
normal form algorithm.

If $C_{1}\neq C_{2}$ or $D_{1}\neq D_{2}$, then $\Box ABC_{1}D_{1}
\not\sim \Box ABC_{2}D_{2}$.

\end{proposition}

\begin{proof} $D_{i} \vartriangleleft C_{i}$, therefore if $C_{1}\neq C_{2}$, then
$\Box ABC_{1}D_{1}\not\sim\Box ABC_{2}D_{2}$.

Suppose $C_{1}=C_{2}$ and $D_{1}\neq D_{2}$. Under any similarity
mapping $C_{1}$ must be mapped to $C_{2}$. If a similarity mapping
fixes three noncollinear points $A$, $B$ and $C_{i}$, then it must
fix any other point of the plane. If $A$, $B$ and $C_{i}$ are on
the $x$-axis, then $D_{i}$ must also be on the $x$-axis and must
be fixed. Therefore $D_{1}\neq D_{2}$ implies $\Box
ABC_{1}D_{1}\not\sim \Box ABC_{2}D_{2}$ in this case.
\end{proof}

\begin{remark} If $C_{s}=D_{s}$, then there
are the following possibilities for the number of similarity types
of quadrilaterals with a given $C\in S_{C}$: 1) $1$ similarity
type if $C=D$, 2) $2$ similarity types if $C\neq D$ and
$|AC|=|BC|$, or $C\neq D$ and $C$ belongs to the $x$-axis or 3)
$4$ similarity types in other cases.

\end{remark}

\section{Possible uses of normal forms in education}

One vertex normal forms of triangles can be used to represent all
similarity types of triangles in a single picture with all
triangles having a fixed side, especially $C$-vertex and
$B$-vertex normal forms. It may be useful to have an example for
students showing that similarity type of triangle can be
parametrized by coordinates of a single point. One vertex normal
forms can also be used in considering quadrilaterals.

The circle normal form of triangles may be useful teaching
properties of circumscribed circles, e.g. inscribing triangles
with given angles in a circle.

Normal forms of triangles can also be used to teach the idea of
normal (canonical) objects using a case of simple and popular
geometric constructions.

\section{Conclusion and further development} It is relatively easy to define several normal forms of
triangles up to similarity. Since the main purpose of this work is
contribution to mathematics education only simplest approaches
which may be used in teaching are considered in this paper. One
approach is to map one side to the $x$-axis and use dilations and
reflections to position the third vertex in a unique way, in this
approach normal triangles are parametrized by one vertex. This
approach can be generalized for quadrilaterals. Another approach
considered in this paper is to design normal triangles as
triangles inscribed in a unit circle. Further development in this
direction may be related to using other figures related to a given
triangle, for example, the inscribed circle, medians, altitudes or
bisectors.


\end{document}